\documentclass[12pt]{amsart}

\usepackage{amssymb, hyperref}

\usepackage[all]{xy}

\theoremstyle{plain}
\newtheorem{thm}{Theorem}
\newtheorem{cor}[thm]{Corollary}
\newtheorem{theorem}{Theorem}[section]
\newtheorem{lemma}[theorem]{Lemma}

\theoremstyle{definition}

\newtheorem{remark}[theorem]{Remark}

\newcommand\bA{{\mathbb A}}

\newcommand\bG{{\mathbb G}}

\newcommand\bP{{\mathbb P}}

\newcommand\bZ{{\mathbb Z}}

\newcommand\cO{{\mathcal O}}

\newcommand\fm{\mathfrak{m}}

\newcommand\charac{{\rm char}}

\renewcommand\deg{{\rm deg}}

\newcommand\id{{\rm id}}

\newcommand\pr{{\rm pr}}

\DeclareMathOperator\Isol{Isol}
\newcommand\Ker{{\rm Ker}}

\newcommand\Spec{{\rm Spec}}
\DeclareMathOperator\Stab{Stab}

\numberwithin{equation}{section}

\title{Homogeneous varieties under split solvable algebraic groups}

\author{Michel Brion}

\address{Universit\'e Grenoble Alpes, 
Institut Fourier, CS 40700, 38058 Grenoble Cedex 9, France}

\date{}

\begin{document}

\begin{abstract}
We present a modern proof of a theorem of Rosenlicht, asserting that
every variety as in the title is isomorphic to a product of affine lines 
and punctured affine lines. 
\end{abstract}

\maketitle

\section{Introduction}
\label{sec:int}

Throughout this note, we consider algebraic groups and varieties over 
a field $k$. An algebraic group $G$ is \emph{split solvable}
if it admits a chain of closed subgroups
\[ \{ e \} = G_0 \subset G_1 \subset \cdots \subset G_n = G \]
such that each $G_i$ is normal in $G_{i + 1}$ and $G_{i+1}/G_i$ is 
isomorphic to the additive group $\bG_a$ or the multiplicative group 
$\bG_m$. This class features prominently in a series of articles 
by Rosenlicht on the structure of algebraic groups, see 
\cite{Ro56, Ro57, Ro63}. The final result of this series may be stated 
as follows (see \cite[Thm.~5]{Ro63}):

\begin{thm}\label{thm:hom}
Let $X$ be a homogeneous variety under a split solvable algebraic 
group $G$. Then there is an isomorphism of varieties
$X \simeq \bA^m \times (\bA^{\times})^n$ 
for unique nonnegative integers $m$, $n$.
\end{thm}

Here $\bA^m \simeq (\bA^1)^m$ denotes the affine $m$-space, and
$\bA^{\times} = \bA^1 \setminus \{ 0\}$ the punctured affine line.

Rosenlicht's articles use the terminology and methods of algebraic geometry
\`a la Weil, and therefore have become hard to read. In view of their 
fundamental interest, many of their results have been rewritten in more 
modern language, e.g. in the book \cite{DG} by Demazure \&  Gabriel
and in the second editions of the books on linear algebraic groups by Borel 
and Springer, which incorporate developments on ``questions of rationality''
(see \cite{Borel, Springer}). The above theorem is a notable exception: 
the case of the group $G$ acting on itself by multiplication is handled in 
\cite[Cor.~IV.4.3.8]{DG} (see also \cite[Cor.~14.2.7]{Springer}), 
but the general case is substantially more complicated. 
\footnote{The case where $k$ is algebraically closed and $X = G/H$ 
for some smooth connected subgroup $H \subset G$ is proposed as
an exercise in \cite[\S 14.2]{Springer}.}

The aim of this note is to fill this gap by providing a proof of 
Theorem \ref{thm:hom} in the language of modern algebraic geometry. 
As it turns out, this theorem is self-improving: combined
with Rosenlicht's theorem on rational quotients (see 
\cite[Thm.~2]{Ro56}, and \cite[Sec.~2]{BGR} for a modern proof)
and some ``spreading out'' arguments, it yields the following stronger 
version:

\begin{thm}\label{thm:op}
Let $X$ be a variety equipped with an action of a split solvable 
algebraic group $G$. Then there exist a dense open $G$-stable 
subvariety $X_0 \subset X$ and an isomorphism of varieties
$X_0 \simeq \bA^m \times (\bA^{\times})^n \times Y$
(where $m$, $n$ are uniquely determined nonnegative integers 
and $Y$ is a variety, unique up to birational isomorphism)
such that the resulting projection $f : X_0 \to Y$ is the rational
quotient by $G$.
\end{thm}

By this, we mean that $f$ yields an isomorphism
$k(Y) \stackrel{\sim}{\to} k(X)^G$, where the left-hand side
denotes the function field of $Y$ and the right-hand side 
stands for the field of $G$-invariant rational functions on $X$; 
in addition, the fibers of $f$ are exactly the $G$-orbits. 

As a direct but noteworthy application of Theorem \ref{thm:op}, 
we obtain:

\begin{cor}\label{cor:rat}
Let $X$ be a variety equipped with an action of a split solvable
algebraic group $G$. Then $k(X)$ is a purely transcendental 
extension of $k(X)^G$.
\end{cor}

When $k$ is algebraically closed, this gives back the main result 
of \cite{Popov}; see \cite{CZ} for applications to the rationality 
of certain homogeneous spaces.

The proof of Theorem \ref{thm:op} also yields a version of 
\cite[Prop.~14.2.2]{Springer}:

\begin{cor}\label{cor:add}
Let $X$ be a variety equipped with a nontrivial action of $\bG_a$.
Then there exist a variety $Y$, an open immersion 
$\varphi : \bA^1 \times Y \to X$ and a monic additive polynomial
$P \in \cO(Y)[t]$ such that 
\[ g \cdot \varphi(x,y) = \varphi(x + P(y,g), y) \]
for all $g \in \bG_a$, $x \in \bA^1$ and $y \in Y$.
\end{cor}

Here $P$ is said to be additive if it satisfies
$P(y,t + u) = P(y,t) + P(y,u)$ identically; then $\bG_a$
acts on $\bA^1 \times Y$ via $g \cdot (x,y) = (x + P(y,g),y)$,
and $\varphi$ is equivariant for this action.
If $\charac(k) = 0$, then we have $P = t$ and hence 
$\bG_a$ acts on $\bA^1 \times Y$ by translation on $\bA^1$. 
So Corollary \ref{cor:add} just means that every nontrivial 
$\bG_a$-action becomes a trivial $\bG_a$-torsor
on some dense open invariant subset. On the other hand, if 
$\charac(k) = p > 0$, then $P$ is a $p$-polynomial, i.e.,
\[ P = a_0 t + a_1 t^p + \cdots + a_n t^{p^n} \]
for some integer $n \geq 1$ and $a_0,\ldots, a_n \in \cO(Y)$. 
Thus, the map 
\[ (P,\id) : \bG_a \times Y \longrightarrow \bG_a \times Y, 
\quad (g,y) \longmapsto (P(y,g),y) \]
is an endomorphism of the $Y$-group scheme 
$\bG_{a,Y} = \pr_Y : \bG_a \times Y \to Y$; conversely,
every such endomorphism arises from an additive polynomial
$P$, see \cite[II.3.4.4]{DG}. Thus, Corollary \ref{cor:add} 
asserts that for any nontrivial $\bG_a$-action, there is a dense 
open invariant subset on which $\bG_a$ acts by a trivial torsor 
twisted by such an endomorphism. These twists occur implicitly 
in the original proof of Theorem \ref{thm:hom}, see 
\cite[Lem.~3]{Ro63}.
\footnote{Rosenlicht was very well aware of the limitations of classical
methods. He wrote in the introduction of \cite{Ro63}: 
``The methods of proof we use here are refinements of those of 
our previous Annali paper \cite{Ro57} and cry for improvement;
there are unnatural complexities and it seems that something new 
that is quite general, and possibly quite subtle, must be brought 
to light before appreciable progress can be made.''}

This note is organized as follows. In Section \ref{sec:ssg},
we gather background results on split solvable algebraic groups.
Section \ref{sec:quot} presents further preliminary material,
on the quotient of a homogeneous space $G/H$ by the left action of 
a normal subgroup scheme $N \triangleleft G$; here $G$ is a connected
algebraic group, and $H \subset G$ a subgroup scheme. In particular,
we show that such a quotient is a torsor under a finite quotient
of $N$, if either $N \simeq \bG_m$ or $N \simeq \bG_a$ and
$\charac(k) = 0$ (Lemma \ref{lem:tors}). The more involved case
where $N \simeq \bG_a$ and $\charac(k) > 0$ is handled in
Section \ref{sec:add}; we then show that the quotient is
a ``torsor twisted by an endomorphism'' as above (Lemma 
\ref{lem:add}). The proofs of our main results are
presented in Section \ref{sec:proofs}. 

\medskip

\noindent
{\bf Notation and conventions.}
We consider schemes over a field $k$ of characteristic $p \geq 0$
unless otherwise mentioned. Morphisms and products of schemes are 
understood to be over $k$ as well. A \emph{variety} is an integral 
separated scheme of finite type.

An \emph{algebraic group} $G$ is a group scheme of finite type.
By a \emph{subgroup} $H \subset G$, we mean a (closed) subgroup scheme.
A $G$-\emph{variety} is a variety $X$ equipped with a $G$-action
\[ \alpha : G \times X \longrightarrow X, 
\quad (g,x) \longmapsto g \cdot x. \]
We say that $X$ is $G$-\emph{homogeneous} 
if $G$ is smooth, $X$ is geometrically reduced, and the morphism 
\[ (\id, \alpha) : G \times X \longrightarrow X \times X,
\quad (g,x) \longmapsto (x, g \cdot x) \]
is surjective. If in addition $X$ is equipped with a $k$-rational point 
$x$, then the pair $(X,x)$ is a $G$-\emph{homogeneous space}.
Then $(X,x) \simeq (G/\Stab_G(x),x_0)$, where 
$\Stab_G(x) \subset G$ denotes the stabilizer, and $x_0$ the image 
of the neutral element $e \in G(k)$ under the quotient morphism
$G \to G/\Stab_G(x_0)$.

Given a field extension $K/k$ and a $k$-scheme $X$, we denote by 
$X_K$ the $K$-scheme $X \times_{\Spec(k)} \Spec(K)$.

We will freely use results from the theory of faithfully flat descent, 
for which a convenient reference is \cite[Chap.~14, App.~C]{GW}.

\section{Split solvable groups}
\label{sec:ssg}

We first recall some basic properties of these groups, taken from  
\cite[IV.4.3]{DG} where they are called ``groupes $k$-r\'esolubles''
(see also \cite[\S 16.g]{Milne}). Every split solvable group 
is smooth, connected, affine and solvable. Conversely, every
smooth connected affine solvable algebraic group over an algebraically
closed field is split solvable (see \cite[IV.4.3.4]{DG}).

Clearly, every extension of split solvable groups is split solvable. 
Also, recall that every nontrivial quotient group of $\bG_m$ is isomorphic 
to $\bG_m$, and likewise for $\bG_a$ (see \cite[IV.2.1.1]{DG}). 
As a consequence, every quotient group of a split solvable group 
is split solvable as well.

We now obtain a key preliminary result (a version of 
\cite[Lem.~1]{Ro63}, see also \cite[Cor.~14.3.9]{Springer}):

\begin{lemma}\label{lem:normal}
Let $G$ be a split solvable group. Then there exists
a chain of subgroups
\[ G_0 = \{ e \} \subset G_1 \subset \cdots G_m 
\subset \cdots \subset G_{m + n } = G,\]
where $G_i \triangleleft G$ for $i = 0, \ldots, m + n$ and 
\[ G_{i+1}/G_i \simeq
\begin{cases} \bG_a & \text{ if } i = 0,\ldots, m - 1,\\
              \bG_m & \text{ if } i = m, \ldots, m + n - 1.\\
\end{cases} \]
\end{lemma}

\begin{proof}
Arguing by induction on $\dim(G)$, it suffices to show that 
either $G$ is a split torus, or it admits a normal subgroup 
$N$ isomorphic to $\bG_a$.

By \cite[IV.4.3.4]{DG}, $G$ admits a normal unipotent subgroup 
$U$ such that $G/U$ is diagonalizable; moreover, $U$ is split
solvable. Since $G$ is smooth and connected, $G/U$ is a split
torus $T$. Also, since every subgroup and every quotient group of
a unipotent group are unipotent, $U$ admits a chain of subgroups 
\[ \{ e \} = U_0 \subset U_1 \subset \cdots \subset U_m = U \]
such that $U_i \triangleleft U_{i+1}$ and $U_{i+1}/U_i \simeq \bG_a$
for any $i = 0, \ldots, m - 1$. By \cite[IV.4.3.14]{DG}, it follows 
that either $U$ is trivial or it admits a central characteristic 
subgroup $V$ isomorphic to $\bG_a^n$ for some integer $n > 0$.
In the former case, $G = T$ is a split torus. In the latter case,
$V \triangleleft G$ and the conjugation action of $G$ on $V$ factors
through an action of $T$. By \cite[Thm.~4.3]{Conrad}, there is a
$T$-equivariant isomorphism of algebraic groups
$V \simeq V_0 \times V'$, where $V_0$ is fixed pointwise by $T$
and $V'$ is a vector group on which $T$ acts linearly. If $V'$ is 
nontrivial, then it contains a $T$-stable subgroup $N \simeq \bG_a$;
then $N \triangleleft G$. On the other hand, if $V'$ is trivial 
then $V$ is central in $G$; thus, every copy of $\bG_a$ in $V$ 
yields the desired subgroup $N$.
\end{proof}

\section{Quotients of homogeneous spaces by normal subgroups}
\label{sec:quot}

Let $G$ be an algebraic group, $H \subset G$ a subgroup,
and $N \triangleleft G$ a smooth normal subgroup. Then 
$H$ acts on $N$ by conjugation. The semi-direct product 
$N \rtimes H$ defined by this action (as in \cite[Sec.~2.f]{Milne}) 
is equipped with a homomorphism to $G$, with schematic image 
the subgroup $NH \subset G$.
Recall that $H \triangleleft NH \subset G$ and 
$NH/H \simeq N/N \cap H$. Denote by 
\[ q : G \longrightarrow G/H, \quad r: G \longrightarrow G/NH \] 
the quotient morphisms. Then $q$ is an $H$-torsor, and hence 
a categorical quotient by $H$. Since $r$ is invariant under the 
$H$-action on $G$ by right multiplication, there exists a unique
morphism $f : G/H \longrightarrow G/NH$
such that the triangle 
\[ \xymatrix{
& G \ar[ld]_{q} \ar[rd]^{r} & \\
G/H \ar[rr]^-{f} & & G/NH \\
} \]
commutes. 

We will also need the following observation (see 
\cite[Prop.~7.15]{Milne}):

\begin{lemma}\label{lem:cart}
With the above notation, the square
\[ \xymatrix{
G \times NH/H \ar[r]^-{a} \ar[d]_{\pr_G} & G/H \ar[d]^{f} \\
G \ar[r]^-{r} & G/NH \\
} \]
is cartesian, where $a$ denotes the restriction of the action 
$G \times G/H \to G/H$ and $\pr_G$ denotes the projection.
\end{lemma}

\begin{proof}
Since $r$ is an $NH$-torsor, we have a cartesian square
\[ \xymatrix{
G \times NH \ar[r]^-{m} \ar[d]_{\pr_G} & G \ar[d]^{r} \\
G \ar[r]^-{r} & G/NH, \\
} \]
where $m$ denotes the restriction of the multiplication
$G \times G \to G$. Also, the square 
\[ \xymatrix{
G \times NH \ar[r]^-{m} \ar[d]_{(\id, q)} & G \ar[d]^{q} \\
G \times NH/H \ar[r]^-{a} & G/H \\
} \]
is commutative, and hence cartesian since the vertical
arrows are $H$-torsors. As $q$ is faithfully flat, this yields
the assertion by descent.
\end{proof}

For simplicity, we set $X = G/H$ and $Y = G/NH$.
These homogeneous spaces come with base points
$x_0$, $y_0$ such that $f(x_0) = y_0$.

\begin{lemma}\label{lem:quot}

\begin{enumerate}

\item[{\rm (i)}] With the above notation, $f$ is $G$-equivariant 
and $N$-invariant, where $G$ (and hence $N$) acts on $X,Y$ 
by left multiplication.

\item[{\rm (ii)}] $f$ is smooth, surjective, and its fibers are 
exactly the $N$-orbits.

\item[{\rm (iii)}] The morphism
\[ \gamma : N \times X \longrightarrow X \times_Y X, \quad
(n,x) \longmapsto (x, n \cdot x) \]
is faithfully flat.

\item[{\rm (iv)}] The map $f^{\#} : \cO_Y \to f_*(\cO_X)$ 
yields an isomorphism 
$\cO_Y \stackrel{\sim}{\to} f_*(\cO_X)^N$,
where the right-hand side denotes the subsheaf of $N$-invariants.

\item[{\rm (v)}] If $N \cap H$ is central in $G$, then 
$f$ is a $N/N \cap H$-torsor.

\end{enumerate}

\end{lemma}

\begin{proof}
(i) Let $R$ be an algebra, $g \in G(R)$ and $x \in X(R)$. 
As $q$ is faithfully flat, there exist a faithfully
flat $R$-algebra $R'$ and $g' \in G(R')$ such that 
$x = g' \cdot x_0$. Then 
$f(g \cdot x) = f(gg' \cdot x_0) = gg' \cdot y_0 
= g \cdot (g' \cdot y_0) = g \cdot f(x)$ 
in $Y(R')$, and hence in $Y(R)$. This yields the 
$G$-equivariance of $f$. 

If $g \in N(R)$ then $gg' = g' n$ for some $n \in N(R')$.
Thus, $f(gg' \cdot x_0) = f(g' \cdot x_0)$, i.e., 
$f(g \cdot x) = f(x)$, proving the $N$-invariance.

(ii) Observe that $NH/H$ is homogeneous under the smooth algebraic 
group $N$, and hence is smooth. Thus, 
$\pr_G: G \times NH/H \to G$ is smooth as well. It follows
that $f$ is smooth by using Lemma \ref{lem:cart} and the 
faithful flatness of $r$. Also, $f$ is surjective since so are 
$\pr_G$ and $r$.

Let $K/k$ be a field extension, $x \in X(K)$, and $y = f(x)$. 
There exist a field extension $L/K$ and $g \in G(L)$ such that 
$x = g \cdot x_0$. Thus, $y = g \cdot y_0$ and the fiber $X_y$ 
satisfies $(X_y)_L = g (X_{y_0})_L$. Also, $X_{y_0} = N \cdot x_0$ 
in view of Lemma \ref{lem:cart} together with the isomorphisms
$N \cdot x_0 \simeq N/N \cap H \simeq NH/H$. Thus,
$(X_y)_L = g \cdot (N x_0)_L = (N g \cdot x_0)_L = (N \cdot x)_L$,
and therefore $X_y = N_K \cdot x$ by descent.

(iii) Consider the commutative triangle
\[ \xymatrix{
N \times X \ar[rr]^-{\gamma} \ar[rd]_{\pr_X} & &
X \times_Y X \ar[ld]^{\pr_1} \\
& X. \\
} \]
Clearly, the morphism $\pr_X$ is faithfully flat. Also, $\pr_1$ 
is faithfully flat, since it is obtained from $f$ by base
change. Moreover, for any field extension $K/k$ and any
$x \in X(K)$, the restriction 
$\gamma_x : N \times x = N_K \to X_x$ is the orbit map
$n \mapsto n \cdot x$, and hence is faithfully flat by (ii). 
So the assertion follows from the fiberwise flatness
criterion (see \cite[IV.11.3.11]{EGA}).

(iv) We have
\[ \cO_Y = r_*(\cO_G)^{NH} = f_* q_*(\cO_G)^{NH}
= f_*(q_*(\cO_G)^H)^N = f_*(\cO_X)^N, \]
since $q$ (resp.~$r$) is a torsor under $H$ (resp.~$NH$).

(v) The subgroup $N \cap H \subset G$ fixes $x_0$ and is
central in $G$. By a lifting argument as in (i), it follows
that $N \cap H$ fixes $X = G \cdot x_0$ pointwise. Thus, 
the $N$-action on $X$ factors uniquely through an action 
of $N/N\cap H$. Since the square 
\[ \xymatrix{
G \times N/N \cap H \ar[r]^-{a} \ar[d]_{\pr_G} & X \ar[d]^{f} \\
G \ar[r]^-{r} & Y \\
} \]
is cartesian (Lemma \ref{lem:cart}) and $r$ is faithfully flat, 
this yields the assertion.
\end{proof}

In view of the assertions (i), (ii), (iii) and (iv), $f$ is 
a geometric quotient by $N$ in the sense of \cite[Def.~0.7]{MFK}.

Next, denote by $\Stab_N \subset N \times X$ the stabilizer,
i.e., the pullback of the diagonal in $X \times_Y X$ 
under $\gamma$. Then $\Stab_N$ is a closed subgroup 
scheme of the $X$-group scheme 
$N_X = (\pr_X : N \times X \to X)$, stable under the $G$-action 
on $N \times X$ via $g \cdot (n,x) = (g n g^{-1}, g \cdot x)$.

\begin{lemma}\label{lem:stab}

\begin{enumerate}

\item[{\rm (i)}] The projection 
$\pr_X :  \Stab_N \to X$ is faithfully flat and $G$-equivariant. 
Its fiber at $x_0$ is $H$-equivariantly isomorphic to $N \cap H$ 
on which $H$ acts by conjugation.

\item[{\rm (ii)}] $\pr_X$ is finite if and only if $N \cap H$ 
is finite.
 
\end{enumerate}

\end{lemma} 
 
\begin{proof}
(i) Clearly, $\pr_X$ is equivariant and its fiber 
$\Stab_N(x_0)$ is as asserted. Form the cartesian square
\[ \xymatrix{
Z \ar[r]^-{\pi} \ar[d] & G \ar[d]^{q} \\
\Stab_N \ar[r]^-{\pr_X} & G/H. \\
} \]
Then $Z$ is equipped with a $G$-action such that
$\pi$ is equivariant, with fiber at $e$ being 
$N \cap H$. As a consequence, the morphism
\[ G \times N \cap H \longrightarrow Z, 
\quad (g,z) \longmapsto g \cdot z \]
is an isomorphism with inverse being
$z \mapsto (\pi(z), \pi(z)^{-1} \cdot z)$.
Via this isomorphism, $\pi$ is identified with
the projection $G \times N \cap H \to G$. Thus,
$\pi$ is faithfully flat, and hence so is $\pr_X$. 

(ii) This also follows from the above cartesian square,
since $\pi$ is finite if and only if $N \cap H$ is finite.
\end{proof}

\begin{lemma}\label{lem:tors}
Assume that $N \not\subset H$.

\begin{enumerate}

\item[{\rm (i)}] If $N \simeq \bG_m$ and $G$ is connected,
then $f$ is an $N/N\cap H$-torsor. Moreover,
$N/ N \cap H \simeq \bG_m$.

\item[{\rm (ii)}] If $N \simeq \bG_a$ and $p = 0$,
then $f$ is an $N$-torsor.

\end{enumerate}

\end{lemma}

\begin{proof}
(i) In view of the rigidity of tori (see \cite[Exp.~IX, Cor.~5.5]{SGA3} 
or \cite[Cor.~12.37]{Milne}), $N$ is central in $G$. Also, $N \cap H$ 
is a finite subgroup of $N$, and hence $N/N \cap H \simeq \bG_m$. 
So we conclude by Lemma \ref{lem:quot} (v). 

(ii) Likewise, $N \cap H$ is a finite subgroup of $\bG_a$,
and hence is trivial since $p = 0$. So we conclude by
Lemma \ref{lem:quot} (v) again.
\end{proof}

\section{Quotients by the additive group}
\label{sec:add}

We first record two preliminary results, certainly well-known
but for which we could locate no appropriate reference.

\begin{lemma}\label{lem:hilb}
Let $X$ be a locally noetherian scheme. Let 
$Z \subset \bA^1 \times X$ be a closed subscheme such that
the projection $\pr_X: Z \to X$ is finite and flat. 
Then $Z$ is the zero subscheme of a unique monic polynomial
$P \in \cO(X)[t]$.
\end{lemma}

\begin{proof}
First consider the case where $X = \Spec(A)$, where 
$A$ is a local algebra with maximal ideal $\fm$ and
residue field $K$. Denoting by $x$ the closed point
of $X$, the fiber $Z_x$ is a finite subscheme of
$\bA^1_K$. Thus, $Z_x = V(P)$ for a unique monic 
polynomial $P \in K[t]$. So the images of 
$1,t, \ldots,t^{n-1}$ in $\cO(Z_x)$ form a basis
of this $K$-vector space, where $n = \deg(P)$.
Also, $\cO(Z)$ is a finite flat $A$-module, hence free.
By Nakayama's lemma, the images of $1,t, \ldots,t^{n-1}$ 
in $\cO(Z)$ form a basis of this $A$-module. So we have
$t^n + a_1 t^{n-1} + \cdots + a_n = 0$ in $\cO(Z)$ for
unique $a_1,\ldots,a_n \in A$. Thus, the natural map
$A[t]/(t^n + a_1 t^{n-1} + \cdots + a_n) \to \cO(Z)$
is an isomorphism, since it sends a basis to a basis.
This proves the assertion in this case.

For an arbitrary scheme $X$, the assertion holds
in a neighborhood of every point by the local case.
In view of the uniqueness of $P$, this completes the proof.
\end{proof}

\begin{lemma}\label{lem:ker}
Let $X$ be a locally noetherian scheme, and 
$H \subset \bG_{a,X}$ a finite flat subgroup scheme.
Then $H = \Ker(P,\id)$ for a unique monic additive polynomial 
$P \in \cO(X)[t]$, where $(P,\id)$ denotes the endomorphism 
\[ \bG_{a,X} \longrightarrow \bG_{a,X}, \quad
(g,x) \longmapsto (P(x,g), x). \]
\end{lemma}

\begin{proof}
We may assume that $X$ is affine by the uniqueness property. 
Let $X = \Spec(A)$, then $H = V(P)$ for a unique monic 
polynomial $P \in A[t]$ (Lemma \ref{lem:hilb}). We now adapt
an argument from \cite[IV.2.1.1]{DG} to show that $P$ is 
an additive polynomial.

Denote by 
$m: \bG_{a,X} \times_X \bG_{a,X} \to \bG_{a,X}$
the group law. Since $H$ is a subgroup scheme, we have 
$H \times_X H \subset m^{-1}(H)$.
Considering the ideals of these closed subschemes of 
$\bG_{a,X} \times_X \bG_{a,X} \simeq 
\bG_a \times \bG_a \times X = \Spec(A[t,u])$
yields that $P(t  + u) \in (P(t),P(u))$ in $A[t,u]$. So there
exist $Q, R \in A[t,u]$ such that
\[ P(t + u) - P(t) - P(u) = Q(t,u) P(t) + R(t,u) P(u). \]
Since $P$ is monic, there exist unique $Q_1,Q_2 \in A[t,u]$
such that 
\[ Q(t,u) = Q_1(t,u) P(u) + Q_2(t,u), \quad 
\deg_u(Q_2) < \deg(P) = n. \]
Thus, we have
\[ P(t + u) - P(t) - P(u) - Q_2(t,u) P(t) = 
(Q_1(t,u) P(t) + R(t,u)) P(u). \]
As the left-hand side has degree in $u$ at most $n-1$, it follows 
that $Q_1(t,u) P(t) + R(t,u) = 0$ and
$P(t + u) - P(t) - P(u) = Q_2(t,u) P(t)$. Considering the
degree in $t$, we obtain $Q_2 = 0$ and $P(t + u) = P(t) + P(u)$
identically. 
\end{proof}

Next, we return to the setting of Section \ref{sec:quot}:
$G$ is an algebraic group, $H \subset G$ a subgroup, 
$N \triangleleft G$ a smooth normal subgroup, and
$f : X = G/H \to G/NH = Y$ the natural morphism.
Since $f$ is $N$-invariant (Lemma \ref{lem:quot} (i)),
we may view $X$ as an $Y$-scheme equipped with an action 
of the $Y$-group scheme $N_Y$.

\begin{lemma}\label{lem:add}
Assume in addition that $N \simeq \bG_a$ and $N \not\subset H$.
Then there exist a faithfully flat morphism of $Y$-group schemes
$\varphi : N_Y \to \bG_{a,Y}$ and a $\bG_{a,Y}$-action on $X$ 
such that $f$ is a $\bG_{a,Y}$-torsor.
\end{lemma}

\begin{proof}
By Lemma \ref{lem:stab}, the stabilizer $\Stab_N$ is finite 
and flat over $Y$. Thus, $\Stab_N = \Ker(P, \id)$
for a unique monic $p$-polynomial $P \in \cO(X)[t]$
(Lemma \ref{lem:ker}). Also,
$\Stab_N \subset N \times X$ is stable under the action of 
the abstract group $N(k)$ via $g \cdot (n,x) = (n, g \cdot x)$;
as a consequence, we have $P(g \cdot x,t) = P(x,t)$
identically on $X$, for any $g \in N(k)$. This still holds after 
base change by a field extension $K/k$, since the formation 
of $\Stab_N$ commutes with such base change and hence
$P$ is invariant under any such extension. Since $N(K)$
is dense in $N_K$ for any infinite field $K$, it follows that
$P$ is $N$-invariant. As $\cO(X)^N = \cO(Y)$ (Lemma 
\ref{lem:quot} (iv)), we see that $P \in \cO(Y)[t]$.

Choose an isomorphism $\bG_a \stackrel{\sim}{\to} N$
and consider the morphism
\[ \varphi = (P, \id) : \bG_{a,Y} \longrightarrow \bG_{a,Y},
\quad (t,y) \longmapsto (P(y,t),y). \]
Then $\varphi$ is an endomorphism of the $Y$-group scheme
$\bG_{a,Y}$. Moreover, $\varphi$ is faithfully flat, as follows
from the fiberwise flatness criterion (see \cite[IV.11.3.11]{EGA}),
since $\bG_{a,Y}$ is faithfully flat over $Y$ and for any $y \in Y$,
the morphism $\varphi_y : t \mapsto P(y,t)$ is faithfully flat.
Denote by $K$ the kernel of $\varphi$. Then we have 
$K \times_Y X = \Stab_N$; thus, $K$ is finite and flat over $Y$,
by Lemma \ref{lem:stab} and descent. Moreover, the square
\[ \xymatrix{
K \times_Y \bG_{a,Y} \ar[r]^-{m} \ar[d] _{\pr} 
& \bG_{a,Y} \ar[d]^{\varphi} \\
\bG_{a,Y} \ar[r]^-{\varphi} & \bG_{a,Y} \\
} \]
is cartesian, where $m$ denotes the group law, and $\pr$
the projection (indeed, $P(t,y) = P(u,y)$ if and only if
$(u-t,y) \in K$). So $\varphi$ is a $K$-torsor. The action
\[ \alpha : \bG_{a,Y} \times_Y X = \bG_a \times X
\longrightarrow X, \quad (t,x) \longmapsto t \cdot x \]
is a $K$-invariant morphism. By descent again, it follows
that there is a unique morphism 
$\beta : \bG_a \times X \to X$ such that the triangle 
\[ \xymatrix{
\bG_a \times X \ar[rr]^-{\alpha} \ar[rd]_{\varphi}  & &
X  \\
& \bG_a \times X \ar[ru]_{\beta} \\
} \]
commutes. Thus,
$\beta(t,x) = \alpha( P(f(x),t),x)$ identically on $\bG_a \times X$.
In particular, $\beta(0,x) = \alpha(0,x) = x$ identically on $X$.
Also, $\beta$ satisfies the associativity property of an action,
since so does $\alpha$ and $\varphi$ is faithfully flat. So 
$\beta$ is an action of $\bG_{a,Y}$ on $X$. Consider the 
associated morphism
\[ \delta : \bG_a \times X \longrightarrow X \times_Y X, 
\quad (t,x) \longmapsto (\beta(t,x),x) \]
as a morphism of $X$-schemes. For any field extension $K/k$
and any $x \in X(K)$, we get a morphism
$\delta_x : \bG_{a,K} \to X_x$ such that 
$\delta_x \circ P_x = \alpha_x$. Thus, $\delta_x$ is an isomorphism
by the construction of $P$. In view of the fiberwise isomorphism
criterion (see \cite[IV.17.9.5]{EGA}), it follows that $\delta$ is
an isomorphism. So $f$ is a $\bG_{a,Y}$-torsor relative to this
action $\beta$.
\end{proof}

\section{Proofs of the main results}
\label{sec:proofs}

\subsection{Proof of Theorem \ref{thm:hom}}
\label{subsec:proofhom}

We first consider the case where $X$ is equipped with a 
$k$-rational point $x_0$. Then $X = G/H$ for some subgroup 
$H \subset G$. If $G$ is a torus, then $G/H$ has the structure 
of a split torus, and hence is isomorphic to $(\bA^{\times})^n$ 
for some integer $n \geq 0$.
Otherwise, $G$ admits a normal subgroup $N \simeq \bG_a$
by Lemma \ref{lem:normal}. If $N \subset H$ then 
$X \simeq (G/N)/(H/N)$ and we conclude by induction on 
$\dim(G)$. So we may assume that $N \not\subset H$. 
Then we have a morphism
\[ f : X = G/H \longrightarrow G/NH \simeq (G/N)/(NH/N). \]
Moreover, $f$ is a $\bG_a$-torsor by Lemma \ref{lem:quot}
(if $p = 0$) and Lemma \ref{lem:add} (if $p > 0$). 
By induction on $\dim(G)$ again, we may assume that 
$Y \simeq \bA^m \times (\bA^{\times})^n$ as a variety. In particular, 
$Y$ is affine, and hence the $\bG_a$-torsor $f$ is trivial. So 
$X \simeq \bA^1 \times Y \simeq \bA^{m + 1} \times (\bA^{\times})^n$ 
as a variety.

To complete the proof, it suffices to show that every homogeneous
$G$-variety has a $k$-rational point. This follows from a result of
Rosenlicht (see \cite[Thm.~10]{Ro56}) and is reproved in 
\cite[Thm.~15.11]{Borel}, \cite[Thm.~14.3.13]{Springer}. 
For completeness, we present a proof based on the following
lemma, also due to Rosenlicht (see \cite[Lem., p.~425]{Ro56}):

\begin{lemma}\label{lem:curve}
Let $X$ be a homogeneous variety under $G = \bG_a$ or
$\bG_m$. Then $X$ has a $k$-rational point.
\footnote{This lemma is reproved in \cite[Prop.~15.6]{Borel}, but
the argument there is unclear to me. In modern language, 
it is asserted that every smooth, geometrically rational curve 
is an open subvariety of a smooth complete curve of genus $0$. 
Yet this fails for nontrivial forms of the affine line, see 
\cite[Lem.~1.1]{Russell}. Also, it is asserted that the $G$-action 
on $X$ extends to an action on its regular completion;
this requires a proof.}
\end{lemma}

\begin{proof}
Since $X$ is a smooth curve, it admits a unique regular completion 
$\bar{X}$, i.e., $\bar{X}$ is a regular projective curve
equipped with an open immersion $X \to \bar{X}$. 
Moreover, $\bar{X}$ is geometrically integral since so is $X$.
We identify $X$ with its image in $\bar{X}$, and denote by
$Z = \bar{X} \setminus X$ the closed complement, equipped
with its reduced subscheme structure. Then 
$Z = \coprod_{i = 1}^n \Spec(K_i)$, where the $K_i/k$
are finite extensions of fields.  

By the smoothness of $X$ again, we may choose a finite separable 
extension $K/k$ such that $X$ has a $K$-rational point $x_0$. 
Then $(X_K,x_0)$ is a homogeneous space under $G_K$, and hence 
is isomorphic to $G_K$ as a variety. Also, $\bar{X}_K$ is 
the regular completion of $X_K$; moreover, $Z_K$ is reduced and 
$\bar{X}_K \setminus X_K = Z_K$. Since $X_K \simeq \bA^1_K$
or $\bA^{\times}_K$, it follows that $\bar{X}_K \simeq \bP^1_K$;
in particular, $\bar{X}$ is a smooth projective curve of genus $0$. 
This identifies $Z_K$ with $\Spec(K)$ (the point at infinity) 
if $G = \bG_a$, resp.~with
$\Spec(K) \coprod \Spec(K) = \{ 0, \infty \}$ if $G = \bG_m$.

In the former case, we have $Z = \Spec(k)$ and hence
$\bar{X}$ has a $k$-rational point. Thus, $\bar{X} \simeq \bP^1$ ,
so that $X$ has a $k$-rational point as well.

In the latter case, let $L = k(X)$; then $L/k$ is separable and
$X_L$ has an $L$-rational point. Thus, we see as above that
$\bar{X}_L \simeq \bP^1_L$ and this identifies $Z_L$ with
$\{ 0, \infty  \}$. In particular, $Z(L) = Z(K)$. Since $K$ and
$L$ are linearly disjoint over $k$, it follows that $Z(k)$
consists of two $k$-rational points; we then conclude as above.
\end{proof}

Returning to a homogeneous variety $X$ under a split solvable
group $G$, we may  choose $N \triangleleft G$ such that 
$N \simeq \bG_a$ or $\bG_m$ (Lemma \ref{lem:normal}). 
Also, we may choose a finite Galois extension $K/k$ such that 
$X$ has a $K$-rational point $x_0$. Let $H = \Stab_{G_K}(x_0)$; 
then $(X_K,x_0)$ is the homogeneous space $G_K/H$, and hence 
there is a geometric quotient 
\[ f : X_K = G_K/H \longrightarrow G_K/N_K H \] 
(Lemma \ref{lem:quot}). Then $f$ is a categorical quotient,
and hence is unique up to unique isomorphism. By Galois
descent (which applies, since all considered varieties are affine),
we obtain a $G$-equivariant morphism $\varphi : X \to Y$ such 
that $\varphi_K = f$. In particular, $Y$ is a homogeneous variety
under $G/N$. Arguing by induction on $\dim(G)$, we may
assume that $Y$ has a $k$-rational point $y$. Then the fiber
$X_y$ is a homogeneous $N$-variety, and hence has 
a $k$-rational point.

\subsection{Proof of Theorem \ref{thm:op}}
\label{subsec:proofop}

We may freely replace $X$ with any dense open $G$-stable 
subvariety. In view of Rosenlicht's theorem on rational 
quotients mentioned in the introduction, we may thus assume 
that there exist a variety $Y$ and a $G$-invariant morphism
\[ f : X \longrightarrow Y \] 
such that $k(Y) \stackrel{\sim}{\to} k(X)^G$ 
and the fiber of $f$ at every $y \in Y$ is a homogeneous
variety under $G_{\kappa(y)}$, where $\kappa(y)$ denotes 
the residue field at $y$. By generic flatness, we may further
assume that $f$ is flat.

Denoting by $\eta$ the generic point of $Y$, the fiber
$X_{\eta}$ is a homogeneous variety under $G_{\eta} = G_{k(Y)}$. 
By Theorem \ref{thm:hom}, this yields an isomorphism
\begin{equation}\label{eqn:gen} 
Z_{\eta} \stackrel{\sim}{\longrightarrow} X_{\eta}, 
\end{equation}
where $Z = \bA^m \times (\bA^{\times})^n$ for unique integers
$m,n \geq 0$. This yields in turn a birational map
\[ \varphi : Z \times Y \dasharrow X \] 
such that $f \circ \varphi = \pr_Y$ as rational maps.

It suffices to show that there exists a dense open 
subvariety $Y_0 \subset Y$ such that $\varphi$ is defined
on $Z \times Y_0$ and yields an open immersion
$Z \times Y_0 \to X$ with $G$-stable image. For this,
we start with some reductions. 

We may assume that $Y$ is affine (by replacing $X$ with
the preimage of a dense open affine subvariety) and also
that $X$ is normal (since its normal locus is a dense 
open $G$-stable subvariety). In view of a result of
Sumihiro (see \cite[Thm.~3.9]{Sumihiro}), we may further assume
that $X$ is a locally closed $G$-stable subvariety of
the projectivization $\bP(V)$, where $V$ is a finite-dimensional
$G$-module. The closure $\bar{X}$ of $X$ in $\bP(V)$ and its 
boundary $\bar{X} \setminus X$ are $G$-stable. By a version
of Borel's fixed point theorem (see \cite[IV.4.3.2]{DG}),
there exist a positive integer $N$ and a nonzero
$s \in H^0(\bar{X},\cO(N))$ which vanishes identically on
$\bar{X} \setminus X$ and is a $G$-eigenvector. Then the
dense open subvariety $\bar{X}_s$ is affine, $G$-stable and
contained in $X$; thus, we may further assume that $X$ is affine.
This replaces $Y$ with a dense open subset $Y_0$ (as $f$
is flat and hence open). As $Y$ is affine, we may choose 
a nonzero $t \in \cO(Y)$ which vanishes identically on 
$Y \setminus Y_0$. Replacing $X$ with $X_t$ and $Y$ with $Y_t$, 
we may finally assume that $X$, $Y$ are affine and $X$ is normal.

Choose a closed immersion of $Y$-varieties 
$X \to \bA^N \times Y$; then $\varphi$ yields a rational map
\[ (\varphi_1,\ldots,\varphi_N,\pr_Y) : 
Z \times Y \dasharrow \bA^N \times Y \] 
such that the pull-back $Z_{\eta} \to \bA^N_{\eta}$ 
is a closed immersion. In particular, 
$\varphi_1,\ldots,\varphi_N \in \cO(Z_{\eta}) 
= \cO(Z) \otimes_k k(Y)$. 
Replacing again $Y$ with a dense open affine subvariety,
we may thus assume that 
$\varphi_1,\ldots,\varphi_N \in \cO(Z) \otimes_k \cO(Y)
= \cO(Z \times Y)$.
As a consequence, $\varphi$ is a morphism.

Denote by $\Isol(\varphi)$ the set of points of $Z \times Y$
which are isolated in their fiber; then $\Isol(\varphi)$
contains the points of $Z_{\eta}$. By Zariski's Main Theorem
(see \cite[III.4.4.3]{EGA}), $\Isol(\varphi)$ is open in 
$Z \times Y$ and the restriction of $\varphi$ to $\Isol(\varphi)$
factors as
\[ \Isol(\varphi) \stackrel{\psi}{\longrightarrow}
 X' \stackrel{\gamma}{\longrightarrow} X,
\]
where $\psi$ is an open immersion and $\gamma$ is finite.
Replacing $X'$ with the schematic image of $\psi$,
we may assume that $\psi$ is schematically dominant; then
$X'$ is a variety. Since $\varphi$ is birational,
so is $\gamma$; as $X$ is normal, it follows that $\gamma$
is an isomorphism. Thus, $\varphi$ restricts to an open
immersion $\Isol(\varphi) \to X$.

Consider the closed complement 
$F = (Z \times Y) \setminus \Isol(\varphi)$. 
Then $F_{\eta}$ is empty, and hence the ideal
$I(F) \subset \cO(Z \times Y)$ satisfies 
$1 \in I(F) \otimes_{\cO(Y)} k(Y)$. Replacing
$Y$ with a principal open subvariety, we may thus
assume that $1 \in I(F)$, i.e., $F$ is empty and
$\Isol(\varphi) = Z \times Y$. Equivalently,
$\varphi : Z \times Y \to X$ is an open immersion.

It remains to show that the image of $\varphi$
is $G$-stable. The isomorphism (\ref{eqn:gen})
is equivariant relative to some action
$\alpha : G_{\eta} \times_{\eta} Z_{\eta} \to Z_{\eta}$.
We may view $\alpha$ as a morphism
$G \times Z \times \eta \to Z$, i.e., a family
$(x_1,\ldots,x_m,y_1,\ldots,y_n)$, where 
$x_1,\ldots,x_m \in \cO(G \times Z \times \eta)$ and 
$y_1,\ldots,y_n \in \cO(G \times Z \times \eta)^{\times}$
(the group of invertible elements). Shrinking $Y$ again,
we may assume that 
$x_1,\ldots,x_m \in \cO(G \times Z \times Y)$ and 
$y_1,\ldots,y_n \in \cO(G \times Z \times Y)^{\times}$.
Then $\alpha$ is given by a morphism 
$G \times Z \times Y \to Z$, i.e., an action of $G_Y$
on $Z \times Y$. Moreover, $\varphi$ is $G_Y$-equivariant,
since so is $\varphi_{\eta}$. This completes the proof 
of Theorem \ref{thm:op}.

The proof of Corollary \ref{cor:add} is completely similar; 
the point is that the generic fiber $X_{\eta}$ is a 
nontrivial $\bG_{a,\eta}$-homogeneous variety, 
and hence is isomorphic to $\bA^1_{\eta}$ on which 
$\bG_{a,\eta}$ acts via a monic additive polynomial 
$P \in k(Y)[t]$ (Lemma \ref{lem:curve}).
We leave the details to the reader.

\begin{remark}\label{rem:final}
(i) Theorem \ref{thm:hom} may be reformulated as follows:
every homogeneous variety $X$ under a split solvable 
algebraic group $G$ is affine and satisfies
\[ \cO(X) \simeq 
k[x_1, \ldots, x_m,y_1, y_1^{-1}, \ldots, y_n, y_n^{-1}], \]
where $x_1, \ldots, x_m, y_1, \ldots, y_n$ are algebraically
independent. So the invertible elements of the algebra
$\cO(X)$ are exactly the Laurent monomials
$c y_1^{a_1} \cdots y_n^{a_n}$, where $c \in k^{\times}$ and
$a_1, \ldots, a_n \in \bZ$. As a consequence, the projection
\[ f: X \longrightarrow (\bA^{\times})^n \] 
is uniquely determined (but the projection $X \to \bA^m$ is not: 
as an example, $k[x,y,y^{-1}] \simeq k[x + P(y),y,y^{-1}]$ for any
$P \in k[t]$). In fact, $f$ is the quotient by the unipotent
part $U$ of $G$, as follows fom the proof of Theorem
\ref{thm:hom}.

\medskip \noindent
(ii) Likewise, in the setting of Theorem \ref{thm:op},
the projection $X_0 \to (\bA^{\times})^n \times Y$ is the rational
quotient by $U$. This theorem is known, in a more
precise formulation, for a variety $X$ equipped with
an action of a connected reductive algebraic group $G$
over an algebraically closed field of characteristic $0$.
Then one considers the action of a Borel subgroup of
$G$, and uses the ``local structure theorem'' as in
\cite[Satz~2.3]{Kn90}. The dimension of $Y$ is the complexity
of the $G$-action on $X$, and $n$ is its rank; both are
important numerical invariants of the action 
(see e.g. \cite[Chap.~2]{Timashev}).

These invariants still make sense in positive characteristics,
and the local structure theorem still holds in a weaker form
(see \cite[Satz~1.2]{Kn93}). Theorem \ref{thm:op} gives additional 
information in this setting.

\medskip \noindent
(iii) Corollary \ref{cor:add} also holds for a variety $X$ equipped
with a nontrivial action of the multiplicative group: there exist
a variety $Y$, a nonzero integer $n$ and an open immersion
$\varphi : \bA^{\times} \times Y \to X$ such that 
$g \cdot \varphi(x,y) = \varphi(g^n x, y)$ identically. This follows 
from the fact that every nontrivial $\bG_{m,\eta}$-homogeneous 
variety is isomorphic to $\bA^{\times}_{\eta}$ on which 
$\bG_{m,\eta}$ acts by the $n$th power map for some $n \neq 0$.

This extends to the action of a split torus $T$: using 
\cite[Cor.~3.11]{Sumihiro}, one reduces to the case where  
$X$ is affine and $T$ acts via a free action of a quotient
torus $T'$. Then the quotient $X \to Y$ exists and is 
a $T'$-torsor, see \cite[Exp.~IX, Thm.~5.1]{SGA3} 
for a much more general result.
\end{remark}

\bibliographystyle{amsalpha}

\end{document}